\documentclass[11pt]{article}

\usepackage{amsfonts,amssymb}

\usepackage{amsthm}

\usepackage{amsmath}

\theoremstyle{plain}

\newtheorem{theorem}{Theorem}[section]
\newtheorem{lemma}[theorem]{Lemma}
\newtheorem{proposition}[theorem]{Proposition}
\newtheorem{corollary}[theorem]{Corollary}
\newtheorem{Counter-example}[theorem]{Counter-example}
\newtheorem{remark}[theorem]{Remark}

\theoremstyle{definition}

\theoremstyle{remark}

\usepackage{amssymb}

\def\R{\mathbb R}

\def\S{\mathbb S}
\def\L{\mathbb L}

\def\E{\mathbb E}

\begin{document}

\title{Compact spacelike surfaces in four-dimensional Lorentz-Minkowski
  spacetime with a non-degenerate \\ lightlike normal direction}

\author{Francisco J. Palomo\footnote{Partially supported by the Spanish MEC-FEDER Grant MTM2010-18099
and the Junta de Andaluc\'{\i}a Regional Grant P09-FQM-4496 with FEDER
funds.}\,\,  and Francisco J. Rodr\'{\i}guez \\ Dep. Matem\'{a}tica Aplicada,
Universidad de M\'{a}laga \\ 29071-M\'{a}laga (Spain)\\
{\ttfamily fjpalomo@ctima.uma.es, fjrodri@ctima.uma.es}
\and
Alfonso Romero$^*$  \\ Dep. Geometr\'{\i}a y Topolog\'{\i}a,
Universidad de Granada \\ 18071-Granada (Spain) \\ {\ttfamily
aromero@ugr.es}
}

\date{ }

\maketitle

\thispagestyle{empty}

\vspace*{-6mm}
\begin{abstract}
A spacelike surface in four-dimensional Lorentz-Minkowski spacetime
through the lightcone has a meaningful lightlike normal vector field
$\eta$. Several sufficient assumptions on such a surface with
non-degenerate $\eta$-second  fundamental form are established to prove
that it must be a totally umbilical round sphere. With this aim, a new
formula which relates the Gauss curvatures of the induced metric and of
the $\eta$-second fundamental form is developed. Then, totally umbilical
round spheres are characterized as the only compact spacelike surfaces
through the lightcone such that its $\eta$-second fundamental form is
non-degenerate and has constant Gauss curvature two.  Another
characterizations of totally umbilical round spheres in terms of the
Gauss-Kronecker curvature of $\eta$ and the area of the $\eta$-second
fundamental form are also given.
\end{abstract}

\hyphenation{me-tric}

\section{Introduction}
The geometry of spacelike surfaces in $4$-dimensional
Lorentz-Minkowski spacetime $\L^4$ through the future lightcone
$\Lambda$ is very rich and appealing. In fact, any $2$-dimensional
simply-connected Riemannian manifold may be isometrically immersed
into $\Lambda$, \cite{LUY}, \cite{PR}. In particular, any
Riemannian metric on the sphere $\S^2$ may be realized as the
induced metric of a spacelike immersion of $\S^2$ in $\L^4$
through $\Lambda$. This situation is quite different when
$\Lambda$ is replaced by a (non-degenerate) hypersurface of
$\L^4$. For instance, there is no spacelike immersion of $\S^2$ in
the unit De Sitter spacetime $\S_1^3 \subset \L^4$ such that the
Gauss curvature of the induced metric satisfies $K>1$, \cite[Cor.
10]{AR}. On the other hand, the existence of an isometric
immersion of an $n(\geq 3)$-dimensional Riemannian manifold in
$\L^{n+2}$ through the corresponding future lightcone has a clear
geometric meaning; namely, an $n(\geq 3)$-dimensional
simply-connected Riemannian manifold $M^n$ admits an isometric
immersion in $\L^{n+2}$ through the lightcone if and only if $M^n$
is conformally flat \cite{AD}, which is a nice characterization of
conformally flatness in terms of Lorentzian geometry.

Motivated in part by these results, spacelike surfaces through the
lightcone in $\L^4$ have been studied from different viewpoints,
\cite{Liu}, \cite{LiuJung}, \cite{LUY}. Focusing our approach
here, if $\psi : M^2 \longrightarrow \L^4$ is a spacelike
immersion such that $\psi(M^2) \subset \Lambda$, the position
vector field $\psi$ is clearly normal and lightlike. The
corresponding Weingarten operator satisfies $A_{\psi}=-I$, where
$I$ denotes the identity transformation, and, therefore, it
provides no information on the extrinsic geometry of $M^2$.
However, there is another lightlike normal vector field $\eta$,
uniquely defined by $\langle \eta, \eta \rangle = 0$ and $\langle
\psi, \eta \rangle = 1$. The Weingarten operator $A_{\eta}$ is
closely related to both intrinsic and extrinsic geometry of $M^2$
by the equations
\[
\mathrm{tr}(A_{\eta})=-K=-\langle \mathbf{H},\mathbf{H} \rangle,
\]
where $K$ is the Gauss curvature of the induce metric and
$\mathbf{H}$ the mean curvature vector field of $\psi$ (Section
2). Moreover, the lightlike normal vector fields $\psi$ and $\eta$
are connected to the so-called $\S^2$-valued Gauss maps
$\mathcal{G}^F$ and $\mathcal{G}^P$, introduced in a more general
context in \cite{Ko}, by
\[
\mathcal{G}^F=\frac{1}{\psi_0}\,\psi \quad \mathrm{and}  \quad
\mathcal{G}^P=-\frac{1}{\eta_0}\,\eta,
\]
where $\psi_0$ and $\eta_0$ are the time coordinates of $\psi$ and
$\eta$, respectively.

\vspace{1mm}

In that follows, given a spacelike surface $M^2$ in $\L^4$ through
$\Lambda$, we will say that $\eta$ is non-degenerate if the
$\eta$-second fundamental form, $\mathrm{II}_{\eta}$, is a
non-degenerate metric on $M^2$. The assumption $\eta$ is
non-degenerate has the following geometric meanings. On the one
hand, $\eta$ is non-degenerate if and only if the Gauss map
$\mathcal{G}^{P}$ is a local diffeomorphism from $M^2$ to
$\S^{2}$. On the other hand, $\eta$ is non-degenerate if and only
if $\widetilde{\psi}:=-\eta$ is also a spacelike immersion (Lemma
\ref{IIII}). In this case, $\widetilde{\psi}$ is said to be the
conjugate immersion to $\psi$. It is remarkable that given a
non-degenerate spacelike surface through $\Lambda$, it is totally
umbilical if and only if its conjugated surface is also totally
umbilical (Corollary \ref{conj}).

\vspace{1mm}

A compact spacelike surface $M^2$ in $\Lambda$ must be
topologically a sphere $\S^2$, \cite{PR}. It is then natural to
wonder for some additional assumption in order to conclude that
$M^2$ is a totally umbilical round sphere. Recall that all the
totally umbilical compact spacelike immersions of $\S^2$ in $\L^4$
through $\Lambda$ were explicitly constructed in \cite{PR} as
follows. If $\psi$ is such a immersion, there exist $u\in \L^{4}$,
$\langle u,u\rangle=-1$, with $u_{0}<0$ and $r>0$, such that,
$$
\psi(\S^{2})=\S^{2}(u,r):=\{\,x\in \L^{4}\; : \; \langle
x,x\rangle=0,\,\,\langle u,x\rangle=r\,\}.
$$
In this case, $A_{\eta}=-(1/2r^{2})I$ and the Riemannian metric on
$M^2$ given by $\mathrm{II}_{\eta}(X,Y):=-\langle
A_{\eta}(X),Y\rangle$ has constant Gauss curvature $K^{\eta}=2$,
\cite{PR}. As shown in \cite[Theor. 5.4]{PR} a compact spacelike
surface $M^{2}$ in $\Lambda$ with constant Gauss curvature is a
totally umbilical round sphere in $\L^{4}$. This result gives an
answer to the previous question from an intrinsic point of view.
In the same philosophy of \cite{AAR}, \cite{Stef} and \cite{AR},
the maim aim of this paper is to obtain several extrinsic
characterizations of the totally umbilical spacelike spheres in
$\L^{4}$ among all the compact spacelike surfaces in $\L^4$ which
factors through $\Lambda$.

\vspace{1mm}

When a spacelike surface $M^2$ in $\Lambda$ is compact, the
non-degeneracy of $\eta$ implies that $\mathrm{II}_{\eta}$ is in
fact Riemannian (Proposition \ref{com}). Our main goal here is
(Theorem \ref{des4}),
\begin{quote}
{\it For a compact spacelike surface $M^2$ of $\L^{4}$ through
$\Lambda$ with $\eta$ non-degenerate, the following assertions are
equivalent:
\begin{enumerate}
\item[{\rm 1.}] $M^2$ is a totally umbilical round sphere,
\item[{\rm 2.}] The Gauss-Kronecker curvature $\mathfrak{d}=\mathrm{det}(A_{\eta})$ is constant,
\item[{\rm 3.}] The Gauss curvature of the Riemannian metric $\mathrm{II}_{\eta}$ satisfies $K^{\eta}=2$.
\end{enumerate}
Moreover, each one of these assumptions is equivalent to the
constancy of the Gauss curvature of the induced metric on $M^2$
{\rm \cite[Theor. 5.4]{PR}}.}
\end{quote}

In order to prove this result, our main tool will be a new formula
which, for any spacelike surface of $\L^4$ through $\Lambda$ with
non-degenerate $\eta$, relates the Gauss curvature $K$ of the
induced metric, the Gauss curvature $K^{\eta}$ of
$\mathrm{II}_{\eta}$ and the Gauss-Kronecker curvature
$\mathfrak{d}$ (Theorem \ref{nuevo}). Note that this extrinsic
quantity is closely related to the notion of quartic curvature
$\mathcal{H}$ of the spacelike surface \cite{Ko}. In fact, it is
not difficult to see that $\mathcal{H}=2\mathfrak{d}$.

\vspace{1mm}

The paper ends with the statement of two equivalent conditions
each one equivalent to each of the three previously stated
(Propositions \ref{des20}, \ref{ultima}):
\begin{quote}
{\it Each of the three equivalent assertions above is equivalent
to,

{\rm 4}. The $\mathrm{II}_{\eta}$-area of $M^2$ satisfies,
$\mathrm{area}(M^2,\mathrm{II}_{\eta})=2\pi$, or

{\rm 5}. The first non-trivial eigenvalue, $\lambda_1$, of the
Laplacian of the induced metric on $M^2$ satisfies,
$$
\lambda_{1}= 2\,\frac{\int_{M^2}\langle \mathbf{H}, \mathbf{H}\rangle \,dA}{\mathrm{area}(M^2,\langle\,,\,\rangle)}\,.
$$}
\end{quote}

\hyphenation{Lo-rent-zi-an}

\section{Preliminaries}
Let $\L^4$ be the Lorentz-Minkowski spacetime, that is, $\R^4$
endowed with the Lorentzian metric, $$\langle \,\, ,\,\, \rangle =
-(dx_{0})^{2}+(dx_{1})^{2}+(dx_{2})^{2}+(dx_{3})^{2},$$ where
$(x_{0},x_{1},x_{2}, x_{3})$ are the canonical coordinates of
$\R^{4}$. A smooth immersion $\psi:M^{2}\rightarrow \L^4$ of a
$2$-dimensional (connected) manifold $M^2$ is said to be a
spacelike if the induced metric via $\psi$ (denoted also by
$\langle\, ,\, \rangle$) is a Riemannian metric on $M^2$.

\vspace{1mm}

Let $\nabla$ and $\overline{\nabla}$ be the Levi-Civita
connections of $M^2$ and $\L^4$, respectively, and let
$\nabla^{\perp}$ be the normal connection. The Gauss
and Weingarten formulas are,
$$\overline{\nabla}_X Y=\psi_{*}(\nabla_XY) + \mathrm{II}(X,Y)
\, \quad \mathrm{and} \, \quad
\overline{\nabla}_X N=-\psi_{*}(A_{N}X)+\nabla^{\perp}_X\,N,$$
for any $X,Y \in \mathfrak{X}(M^{2})$ and $N \in
\mathfrak{X}^{\perp}(M^{2})$, where $\mathrm{II}$ denotes the
second fundamental form of $\psi$. The shape (or Weingarten) operator $A_N$
corresponding to $N$ is related to $\mathrm{II}$ by,
$$\langle A_{N}X, Y \rangle = \langle \mathrm{II}(X,Y), N
\rangle,$$ for all $X,Y \in \mathfrak{X}(M^{2})$. The mean
curvature vector field of $\psi$ is given by
$\mathbf{H}=\frac{1}{2}\mathrm{tr}_{\langle\,\,,\,\,\rangle}\mathrm{II}$.
For each $N\in \mathfrak{X}^{\perp}(M^{2})$, the Codazzi equation
gives,
\begin{equation}\label{Cod}
(\nabla_{X}A_{N})Y-(\nabla_{Y}A_{N})X=
A_{\nabla^{\perp}_{X}N}Y - A_{\nabla^{\perp}_{Y}N}X.
\end{equation}
We denote by $\mathrm{II}_{N}$ the symmetric tensor field on $M^2$
defined by,
$$\mathrm{II}_{N}(X,Y)=-\langle A_{N}X, Y \rangle.$$
We will call $\mathrm{det}(A_N)$ the Gauss-Kronecker curvature of
$M^2$ with respect to the normal vector field $N$. The normal
vector field $N$ is said to be non-degenerate whenever
$\mathrm{det}(A_N)\neq 0 $ at every point $p\in M^{2}$,
\cite{Chen}. When $N$ is non-degenerate, $\mathrm{II}_{N}$ is a
semi-Riemannian metric on $M^2$.

\vspace{1mm}

We write, $$\Lambda = \{\,v\in \L^{4}\,:\,\langle v,v \rangle
=0,\, v_{0}>0\,\},$$ for the future lightcone of $\L ^4$. A
spacelike surface $\psi:M^{2}\rightarrow \L^4$ factors through the
lightcone if $\psi(M^{2})\subset \Lambda$. Every spacelike surface
in $\L^4$ which factors through the lightcone must be orientable
\cite[Lemma 3.2]{PR}. Therefore, we can globally take a lightlike
vector field $\eta\in \mathfrak{X}^{\perp}(M^{2})$ with $\langle
\psi,\eta\rangle=1$.

\vspace{1mm}

From now on, unless otherwise was stated, we will assume
$\psi:M^{2}\to \L^{4}$ is a spacelike surface which factors
through the lightcone. Recall briefly several local geometric
properties of such a surface. Proofs for these features can be
found in \cite{PR}. First, the lightlike normal vector fields
$\psi$ and $\eta$ are parallel with respect to the normal
connection. The corresponding Weingarten operators are given by,
\begin{equation}\label{formulo}
A_{\psi}=-I, \qquad A_{\eta}= -\frac{1+\| \nabla \psi_{0}\|^2
}{2\psi_{0}^{2}}I+\frac{1}{\psi_{0}}\,\nabla^{2}
\psi_{0},
\end{equation}
where $\nabla^{2} \psi_{0} (v)=\nabla_{v}(\nabla \psi_{0})$ for
every $v\in T_pM^{2}$, $p\in M^2$ and  we have written $\psi_{0}$
for $x_{0}\circ \psi$. Recall that the Gauss curvature for the
induced metric on $M^2$ satisfies,
\begin{equation}\label{traza}
K=-\mathrm{tr}(A_{\eta})=\langle \mathbf{H},\mathbf{H}\rangle,
\end{equation}
and therefore, the second fundamental form satisfies,
$$\langle \mathrm{II},\mathrm{II} \rangle(p)=\sum_{i,j=1}^2\langle \mathrm{II}(e_{i},e_{j}),\mathrm{II}(e_{i},e_{j})\rangle=2K(p),$$ where $\{e_{1},e_{2}\}$ is an orthonormal basis of $T_{p}M^2$, $p\in M^2$.

\vspace{1mm}

We write $\mathfrak{d}=\mathrm{det}(A_{\eta})$ for the
Gauss-Kronecker curvature with respect to $\eta$. From formula (\ref{traza}), we arrive to the following technical
result which will be useful along this paper.

\begin{lemma}\label{des2}
Let $\psi : M^2 \longrightarrow \L^{\,4}$ be a spacelike immersion
which factors through the lightcone $\Lambda$. Then,
\begin{equation}\label{des50}
4\,\mathfrak{d}\leq K^{2}\leq 2\, \mathrm{tr}(A_{\eta}^{2}),
\end{equation}
and one equality holds (if and only if the other one also holds)
if and only if $M^{2}$ is totally umbilical.
\end{lemma}

\begin{remark}
{\rm Note that if $M^2$ is assumed to be compact, there exists
$p_{0}\in M^2$ such that equalities (\ref{des50}) hold at $p_0$.
In fact, otherwise we can define two smooth functions $f_1$ and
$f_2$ on $M^2$ with $f_1 < f_2$ and $\{f_{1}(p),f_{2}(p)\}$ are
the eigenvalues of $A_{\eta}$ at every point $p\in M^2$. Each one
of the eigendirections provides a $1$-dimensional distribution on
$M^2$. Since $M^2$ must be a topological sphere $\S^2$, this is
not possible.

On the contrary, the situation for noncompact complete spacelike
surfaces is completely different. Consider the following isometric
immersion $\psi$ of the Euclidean plane $\E^2$ in $\L^4$ through
the lightcone, $$\psi(x,y)=(\cosh x,\sinh x, \cos y,\sin y),$$
$(x,y)\in \E^2$. The lightlike normal vector field $\eta$ is given by
$\eta(x,y)=\frac{1}{2}(-\cosh x,-\sinh x, \cos y,\sin y)$. A
direct computation shows
$A_{\eta}(\partial_{x})=-(1/2)\,\partial_{x}$ and
$A_{\eta}(\partial_{y})=(1/2)\,\partial_{y}$. Therefore,
$\mathfrak{d}=-1/4$, $K=0$ and $2\, \mathrm{tr}(A_{\eta}^{2})=1$.}
\end{remark}

As a direct consequence of (\ref{formulo}) we get.

\begin{proposition}\label{com}
Let $\psi : M^2 \to \L^{\,4}$ be a spacelike immersion which
factors through the lightcone $\Lambda$. If the function
$\psi_{0}$ attains a maximum value at $p_{0}\in M^2$, then
$\mathrm{II}_{\eta}$ is positive definite in a neighborhood of
$p_{0}$. In particular, if $M^2$ is compact and $\eta$ is
non-degenerate, $\mathrm{II}_{\eta}$ is a Riemannian metric.
\end{proposition}

\begin{remark}\label{curpositiva}
{\rm (a) If $\mathrm{II}_{\eta}$ is a Riemannian metric on $M^2$,
we get from (\ref{traza}) that $K>0$ on all $M^2$. (b) For a
noncompact complete spacelike surface we can have even
$A_{\eta}\equiv 0$. In fact, consider the isometric immersion
$\phi$ of the Euclidean plane $\E^{2}$ in $\L^{4}$ given by,
$$\phi(x,y)=\Big(\frac{x^2+y^2+1}{2},\frac{x^2+y^2-1}{2},x,y\Big).$$
Clearly $\phi(\E^2)\subset \Lambda$, $\eta(x,y)=(-1,-1,0,0)$ and
therefore $A_{\eta}=0$ at every point $(x,y)\in \E^2$. }
\end{remark}

For every spacelike immersion $\psi : M^2 \to \L^{\,4}$ which
factors through the lightcone $\Lambda$, we consider the smooth
map $\widetilde{\psi}:M^{2}\to \Lambda$ given by $
\widetilde{\psi}=-\eta. $ In general, $\widetilde{\psi}$ fails to
be an immersion (see previous Remark). In fact, for every $v\in
T_{p}M^2$, we have that
$\widetilde{\psi}_{*}(v)=-\overline{\nabla}_{v}\eta=\psi_{*}(A_{\eta}(v))$.
Hence we get the following result.

\begin{lemma}\label{IIII}
Let $\psi : M^2 \to \L^{\,4}$ be a spacelike immersion which
factors through the lightcone $\Lambda$. Then, $\widetilde{\psi}$
is an immersion if and only if $\eta$ is non-degenerate. In this
case, the induced metric from $\widetilde{\psi}$ is Riemannian and
agrees with the third fundamental form corresponding to $\eta$,
that is,
$$
\widetilde{\psi}^{*}\langle u,v\rangle=\langle A_{\eta}^2(u), v \rangle,
$$
for every $u,v \in T_{p}M^2$, $p\in M^2$.
\end{lemma}

If $\eta$ is assumed to be non-degenerate, we will write
$\widetilde{\psi}^{*}\langle
\,\,\,,\,\,\,\rangle=\mathrm{III}_{\eta}$ and $\widetilde{\eta}$
will denote the lightlike normal vector field coresponding to $\widetilde{\psi}$. Observe
that $\widetilde{\eta}=-\psi$, in particular
$\psi=\widetilde{\widetilde{\psi}}$. The Weingarten operators and
the second fundamental form for $\widetilde{\psi}$ will
represented by $\widetilde{A}$ and $\widetilde{\mathrm{II}}$,
respectively. Note that, in general,
$\widetilde{A}_{\widetilde{\eta}}=\widetilde{A}_{-\psi}\neq I$.
The spacelike surface $\widetilde{\psi}:M^{2}\to \Lambda$ is
called the conjugated surface to $\psi$.

\begin{proposition}\label{des30}
Let $\psi : M^2 \to \L^{\,4}$ be a spacelike immersion which
factors through the lightcone $\Lambda$. Assume $\eta$ is
non-degenerate. Then we have,
\begin{enumerate}
\item[{\rm 1.}] $\widetilde{A}_{\widetilde{\eta}}=A_{\eta}^{-1}$.
\item[{\rm 2.}] $\widetilde{\mathrm{II}}_{\widetilde{\eta}}=\mathrm{II}_{\eta}$.
\end{enumerate}
\end{proposition}
\begin{proof}
The Weingarten equation for $\widetilde{\psi}$ can be written as
follows,
$$\overline{\nabla}_{v}\widetilde{\eta}=-\widetilde{\psi}_{*}\Big(\widetilde{A}_{\widetilde{\eta}}(v)\Big)=-\psi_{*}\Big(A_{\eta}(\widetilde{A}_{\widetilde{\eta}}(v))\Big),$$
for every $v\in T_{p}M^2$, $p\in M^2$. On the other hand,
$\overline{\nabla}_{v}\widetilde{\eta}=-\psi_{*}(v)$ and we deduce
the first assertion. Now, for $u,v\in T_{p}M^2$,
$$
\widetilde{\mathrm{II}}_{\widetilde{\eta}}(u,v)=-\langle A^{2}_{\eta}(\widetilde{A}_{\widetilde{\eta}}(u)),v\rangle=\mathrm{II}_{\eta}(u,v).
$$
\end{proof}

\begin{corollary}
Let $\psi : M^2 \to \L^{\,4}$ be a spacelike immersion which
factors through the lightcone $\Lambda$. If $\eta$ is
non-degenerate, then,
\begin{equation}\label{III}
K^{\mathrm{III}_{\eta}}=\frac{K}{\mathfrak{d}}.
\end{equation}
\end{corollary}
\begin{proof}
This easily follows from previous result taking into account
(\ref{traza}) for $\widetilde{\psi}$.
\end{proof}

Now, Lemma \ref{des2} and Proposition \ref{des30} give.
\begin{corollary}\label{conj}
Let $\psi : M^2 \longrightarrow \L^{\,4}$ be a spacelike immersion
which factors through the lightcone $\Lambda$. Assume $\eta$ is
non-degenerate. Then $\psi$ is totally umbilical if and only if
$\widetilde{\psi}$ is totally umbilical.
\end{corollary}

\begin{remark}
{\rm An interesting question on spacelike surfaces which factor
through the lightcone is the  behavior under the effect of
expansions. That is, let $\psi : M^2 \longrightarrow \L^{\,4}$ be
a spacelike immersion which factors through the lightcone
$\Lambda$. For every $\sigma\in C^{\infty}(M^2)$, consider the
inmersion $\psi_{\sigma}:=e^{\sigma} \psi$. Clearly,
$\psi_{\sigma}$ factors through the lightcone and
$g_{\sigma}=\psi_{\sigma}^{*}\langle\,\,\,,\,\,\,\rangle=e^{2\sigma}\langle\,\,\,,\,\,\,\rangle$,
where as usual we have written
$\psi^{*}\langle\,\,\,,\,\,\,\rangle=\langle\,\,\,,\,\,\,\rangle$.
Therefore, $\psi_{\sigma}$ is spacelike and the Gauss curvature
$K_{\sigma}$ of $g_{\sigma}$ satisfies,
\begin{equation}\label{conf}
K_{\sigma}=\frac{K-\triangle \sigma}{e^{2\sigma}}.
\end{equation}

Let $\eta_{\sigma}=e^{-\sigma}\eta$ be the lightlike normal vector
field such that $\langle \psi_{\sigma},\eta_{\sigma}\rangle=1$. A
straightforward computation from (\ref{formulo}) gives that for
every $X\in \mathfrak{X}(M^2)$,
$$
A^{\sigma}_{\eta_{\sigma}}(X)=\frac{1}{e^{2\sigma}}\Big(A_{\eta}(X)+\nabla_{X}\nabla\sigma+\frac{\|\nabla \sigma\|^2}{2}X-X\sigma\cdot \nabla \sigma\Big),
$$
where $A^{\sigma}_{\eta_{\sigma}}$ denotes the Weingarten operator
corresponding to $\eta_{\sigma}$ with respect to the spacelike
inmersion $\psi_{\sigma}$. Observe that (\ref{conf}) also achieves
from $K_{\sigma}=-\mathrm{tr}(A^{\sigma}_{\eta_{\sigma}})$. Assume
now $A^{\sigma}_{\eta_{\sigma}}$ is non-degenerate, then
$$
\mathrm{II}^{\sigma}_{\eta_{\sigma}}=\mathrm{II}_{\eta}+d\sigma\otimes d\sigma-\frac{\|\nabla \sigma\|^2}{2}\langle\,\,\,,\,\,\,\rangle-\mathrm{Hess}^{\sigma}.
$$
In particular, if $\sigma$ is a constant $
\mathrm{II}^{\sigma}_{\eta_{\sigma}}=\mathrm{II}_{\eta}$. The
converse holds in the compact case. In fact, from
$\mathrm{II}^{\sigma}_{\eta_{\sigma}}=\mathrm{II}_{\eta}$ we get
$\triangle \sigma=0$.}
\end{remark}

\section{The Gauss curvature of $\mathrm{II}_{\eta}$}
Assume now $\mathrm{II}_{\eta}$ is a Riemannian metric on $M^2$.
This section is devoted to obtain a formula which relates the
Gauss curvature $K$ of the induced Riemannian metric $\langle\,\,
, \,\,\rangle$ and the Gauss curvature $K^{\eta}$ of the metric
$\mathrm{II}_{\eta}$.

Let $\nabla^{\mathrm{II}_{\eta}}$ be the Levi-Civita connection of the metric
$\mathrm{II}_{\eta}$. The difference tensor $L$ between the
Levi-Civita connections $\nabla^{\mathrm{II}_{\eta}}$ and $\nabla$ is the symmetric tensor
given by,
$$
L(X,Y)=\nabla^{\mathrm{II}_{\eta}}_{X}Y-\nabla_{X}Y,
$$
for all $X,Y\in \mathfrak{X}(M^{2})$. The Koszul formula \cite[p.
61]{One83} for $\mathrm{II}_{\eta}$, the Codazzi equation
(\ref{Cod}) and $\nabla^{\perp}\eta=0$ show,
\begin{equation}\label{L}
L(X,Y)=\frac{1}{2}A_{\eta}^{-1}\Big[(\nabla_{X}A_{\eta})Y\Big].
\end{equation}

The Riemannian curvature tensor $R^{\eta}$ of $\mathrm{II}_{\eta}$
is obtained as, $$R^{\eta}=R\,+\,Q_{1}\,+\,Q_{2},$$ where $R$
is the Riemannian curvature tensor of the induced metric and
\[
Q_{1}(X,Y)Z=(D_{X}L)(Y,Z)-(D_{Y}L)(X,Z),
\]
\[
Q_{2}(X,Y)Z=L(Y,L(X,Z))- L(X,L(Y,Z)),
\]
for all $X,Y,Z\in \mathfrak{X}(M^{2})$. Therefore the Gauss
curvature $K^{\eta}$ satisfies,
\begin{equation}\label{main}
2K^{\eta}=
\mathrm{tr}_{\mathrm{II}_{\eta}}(\mathrm{Ric})+\mathrm{tr}_{\mathrm{II}_{\eta}}(\widehat{Q_{1}})
+\mathrm{tr}_{\mathrm{II}_{\eta}}(\widehat{Q_{2}}),
\end{equation}
where $\widehat{Q_{i}}(X,Y)=\mathrm{tr}\{Z \mapsto
Q_{i}(Z,X)Y\}$ and $\mathrm{tr}_{\mathrm{II}_{\eta}}$
stands for the trace of the $(1,1)$-tensor
$\overline{\mathcal{T}}$ defined by
$\mathrm{II}_{\eta}(\overline{\mathcal{T}}(X),Y)=\mathcal{T}(X,Y)$.

\begin{theorem}\label{nuevo}
Let $\psi : M^2 \longrightarrow \L^{\,4}$ be a spacelike immersion
which factors through the lightcone $\Lambda$ such that
$\mathrm{II_{\eta}}$ is a Riemannian metric. Then,
\begin{equation}\label{des5}
2K^{\eta}= \frac{K^{2}}{\mathfrak{d}}+\mathrm{II}_{\eta}(L,L)-\frac{1}{4\mathfrak{d}^{2}}\,\mathrm{II}_{\eta}(\nabla^{\mathrm{II}_{\eta}}\mathfrak{d},\nabla^{\mathrm{II}_{\eta}}\mathfrak{d}).
\end{equation}
\end{theorem}
\begin{proof}
Fix $p\in M^{2}$ and let $\{e_{1},e_{2}\}$ be an orthonormal basis
of $T_{p}M^{2}$ for $\langle \,\, ,\,\,\rangle$ which satisfies
$A_{\mathcal{\eta}}(e_{i})=-\lambda_{i}e_{i}$ with $\lambda_{i}>0
$ for $i=1,2$. Then $\{w_{1},w_{2}\}$, where
$w_{i}=(\lambda_{i})^{-1/2}e_{i}$, is an orthonormal basis for
$\mathrm{II}_\mathcal{\eta}$. Taking into account (\ref{traza}), a
direct computation shows that,
$$
\mathrm{tr}_{\mathrm{II}_{\eta}}(\mathrm{Ric})=\frac{K^2}{\mathfrak{d}}.
$$

From (\ref{L}), we obtain that $\mathrm{II}_{\eta}(L(X,Y),Z)$ is
symmetric in all three variables and therefore,
$$
\mathrm{II}_{\eta}(Q_{1}(X,Y)Y,X)=\mathrm{II}_{\eta}(Q_{1}(X,Y)X,Y).
$$
Now, it is easily deduced that,
$$
\mathrm{tr}_{\mathrm{II}_{\eta}}(\widehat{Q_{1}})=0.
$$

Taking into account (\ref{L}), we obtain,
$$
\mathrm{II}_{\eta}(L(X,Y),Z)=\mathrm{II}_{\eta}(L(X,Z),Y),
$$
for every $X,Y,Z \in \mathfrak{X}(M^{2})$.
A straightforward computation shows,
$$
\mathrm{tr}_{\mathrm{II}_{\eta}}(\widehat{Q_
{2}})=2\Big(\mathrm{II}_{\eta}(L(w_{1},w_{2}),L(w_{1},w_{2}))-\mathrm{II}_{\eta}(L(w_{1},w_{1}),L(w_{2},w_{2}))\Big)
$$
$$=\mathrm{II}_{\eta}(L,L)-\mathrm{II}_{\eta}(\mathrm{tr}_{\mathrm{II}_{\eta}}(L),\mathrm{tr}_{\mathrm{II}_{\eta}}(L)),$$
where $$
\mathrm{II}_{\eta}(L,L)=\sum_{i,j}\mathrm{II}_{\eta}\big(L(w_{i},w_{j}),L(w_{i},w_{j})\big),
$$
and
$$\mathrm{tr}_{\mathrm{II}_{\eta}}(L)=-L(w_{1},w_{1})-L(w_{2},w_{2}),$$
denotes the vector field obtained from the
$\mathrm{II}_{\eta}$-contraction of $L$.

We end the proof with an explicit expression of
$\mathrm{tr}_{\mathrm{II}_{\eta}}(L)$. Let $\{E_{1},E_{2}\}$ be a
local orthonormal frame for $\langle \,\, ,\,\,\rangle$ which
satisfies $A_{\mathcal{\eta}}(E_{i})=-f_{i}E_{i}$ for smooth
functions $f_{i}>0$, $i=1,2$ (see comment in \cite[p. 1815]{AAR}).
Then, $\{W_{1},W_{2}\}$, where $W_{i}=(f_{i})^{-1/2}E_{i}$, is a
local orthonormal frame for $\mathrm{II}_\mathcal{\eta}$. Recall
that $\nabla^{\perp}\eta=0$. Now a direct computation shows that,
$$
\langle (\nabla_{X}A_{\eta})E_{i},E_{i}\rangle=X(f_{i}),
$$
for any $X\in \mathfrak{X}(M^2)$,
and the Codazzi equation implies,
\begin{equation}\label{33}
X(f_{i})=\langle (\nabla_{E_{i}}A_{\eta})E_{i},X\rangle.
\end{equation}
Finally, from (\ref{33}) and (\ref{L}) we obtain,
$$
X(\log \mathfrak{d})=\langle (\nabla_{W_{1}}A_{\eta})W_{1},X\rangle+\langle (\nabla_{W_{2}}A_{\eta})W_{2},X\rangle
$$
$$
=-\mathrm{II}_{\eta}\Big(A_{\eta}^{-1}\Big(\sum_{j=1}^{2} (\nabla_{W_{j}}A_{\eta})W_{j}\Big),X\Big)=2\mathrm{II}_{\eta}(\mathrm{tr}_{\mathrm{II}_{\eta}}(L),X).
$$
Therefore,
\begin{equation}\label{calor}
\mathrm{tr}_{\mathrm{II}_{\eta}}(L)=\frac{\nabla^{\mathrm{II}_{\eta}}\mathfrak{d}}{2\mathfrak{d}},
\end{equation}
which completes the proof.
\end{proof}

\begin{remark}{\rm An alternative proof of this formula
can be achieved, using a local computation, from \cite[Exercise
I.18]{Eisenhart}; compare with \cite[Proppsition 3.4 ]{AR}. On the
other hand, a key fact in order to get formula (\ref{des5}) has
been $\nabla^{\perp}\eta=0$. For an arbitrary spacelike surface in
$\L^4$, every lightlike normal vector field $\eta$ must be
recurrent. That is, we have $\nabla^{\perp}\eta=\omega \otimes
\eta$ where $\omega$ is a one form on $M^2$. If in addition,
$\eta$ is assumed to be non-degenerate, a formula relating the
Gauss curvatures $K$ and $K^{\eta}$ can be also obtained as a wide
extension of (\ref{des5}).}
\end{remark}

\begin{proposition}\label{des7}
Let $\psi : M^2 \to \L^{\,4}$ be a spacelike immersion which
factors through the lightcone $\Lambda$ and assume
$\mathrm{II_{\eta}}$ is a Riemannian metric. Then, $M^{2}$ is
totally umbilical in $\L^{4}$ if and only if the Gauss-Kronecker
curvature with respect to $\eta$ is a constant and $K^{\eta}=2$.
\end{proposition}
\begin{proof}
Assume $A_{\eta}=\lambda\, I$, where $\lambda\in \R$. Then (\ref{des5}) reduces to,
$$
2K^{\eta}=\frac{K^{2}}{\lambda^{2}}.
$$
Now from (\ref{traza}) we get that $K^{\eta}=2$. For the converse, note that (\ref{des5}) implies that
$
K^{2}\leq 4\,\mathfrak{d},
$
and Lemma \ref{des2} applies to end the proof.
\end{proof}

\begin{remark}\label{333}
{\rm From Proposition \ref{des7} and Remark \ref{curpositiva},
every complete spacelike surface $M^2$ in the lightcone $\Lambda$
with $\mathrm{II_{\eta}}$ a Riemannian metric and totally
umbilical satisfies $K=2|\lambda|>0$. As a consequence of the
classical Myers theorem, if we assume $M^2$ geodesically complete,
$M^2$ must be compact, and hence a round sphere. }
\end{remark}

\section{Main results}
For compact surfaces, Proposition \ref{des7} can be improved as
the following result states.
\begin{theorem}\label{des4}
Let $\psi : M^2 \to \L^{\,4}$ be a compact spacelike immersion
which factors through the lightcone $\Lambda$.  Assume $\eta$ is
non-degenerate. Then the following conditions are equivalent:
\begin{enumerate}
\item[{\rm 1.}] $M^2$ is a totally umbilical round sphere,
\item[{\rm 2.}] The Gauss-Kronecker curvature $\mathrm{det}(A_{\eta})$ is constant,
\item[{\rm 3.}] The Gauss curvature of the Riemannian metric $\mathrm{II}_{\eta}$ satisfies $K^{\eta}=2$.
\end{enumerate}
\end{theorem}
\begin{proof}
From Proposition \ref{com}, the metric $\mathrm{II}_{\eta}$ is
Riemannian. If we assume $M^2$ is totally umbilical, Proposition
\ref{des7} gives that $\mathfrak{d}$ is constant and $K^{\eta}=2$.
Assume now the Gauss-Kronecker curvature $\mathfrak{d}$ is
constant. Since $M^2$ is a topological $2$-sphere \cite{PR}, we
have $\mathfrak{d}>0$. Therefore, Lemma \ref{des2} assures that
$K\geq 2\sqrt{\mathfrak{d}}$, with equality if and only if $M^2$
is totally umbilical. Now from Theorem \ref{nuevo},
\begin{equation}\label{des3}
2K^{\eta} \geq \frac{K^{2}}{\mathfrak{d}}\geq \frac{2K}{\sqrt{\mathfrak{d}}}.
\end{equation}
The area elements corresponding to $\langle\,,\,\rangle$ and
$\mathrm{II}_{\eta}$ are related by
$dA_{\mathrm{II}_{\eta}}=\sqrt{\mathfrak{d}}\,dA_{\langle\,,\,\rangle}$.
Hence the Gauss-Bonnet formula and (\ref{des3}) imply,
$$
8\pi =\int_{M^{2}}2K^{\eta}\,dA_{\mathrm{II}_{\eta}}\geq 2\int_{M^{2}}\frac{K}{\sqrt{\mathfrak{d}}}\,dA_{\mathrm{II}_{\eta}}=2\int_{M^{2}}K\,dA_{\langle\,,\,\rangle}=8\pi.
$$
We get the equality in (\ref{des3}) and so
$K=2\sqrt{\mathfrak{d}}$ and $K^{\eta}=2$. Finally, under the
assumption $K^{\eta}=2$, Lemma \ref{des2} can be rewritten as
follows: $ K^{\eta}\sqrt{\mathfrak{d}}\leq K, $ again equality
holds if and only if $M^2$ is totally umbilical. From the
Gauss-Bonnet formula,
$$
4\pi=\int_{M^2}K^{\eta}\,dA_{\mathrm{II}_{\eta}}=\int_{M^2}K^{\eta}\sqrt{\mathfrak{d}}\,dA_{\langle\,,\,\rangle}\leq \int_{M^2}K\,dA_{\langle\,,\,\rangle}=4\pi,
$$
and $K^{\eta}\sqrt{\mathfrak{d}}= K$.
\end{proof}

\begin{remark}
{\rm A compact spacelike immersion $\psi$ which factors through
the lightcone $\Lambda$ with constant Gauss curvature must be
totally umbilical  \cite[Theorem 5.4]{PR}. From (\ref{III}) and
Corollary \ref{conj} it follows that $\psi$ is totally umbilical
if and only if $K/\mathfrak{d}$ is a constant.}
\end{remark}

We end the paper with the statement of two results which
complement Theorem \ref{des4} from points of view.

\begin{proposition}\label{des20}
Let $\psi : M^2 \to \L^{\,4}$ be a compact spacelike immersion
which factors through the lightcone $\Lambda$.  Assume $\eta$ is
non-degenerate. Then,
$$
\mathrm{area}(M^2, \mathrm{II}_{\eta})\leq 2\pi,
$$
and equality holds if and only if $M^2$ is totally umbilical.
\end{proposition}
\begin{proof}
In Theorem \ref{des4} we have pointed out that
$dA_{\mathrm{II}_{\eta}}=\sqrt{\mathfrak{d}}\,dA_{\langle\,,\,\rangle}$
and  $2\sqrt{\mathfrak{d}}\leq K$. Therefore, the result follows
as a consequence of the Gauss-Bonnet formula and Lemma \ref{des2}.
\end{proof}

For a compact submanifold $M^n$ of an Euclidean space $\E^{n+p}$
there is a well-known upper bound of the first non-trivial
eigenvalue $\lambda_1$ of the Laplacian of $M^n$ called the
classical Reilly formula \cite{Re}. This upper bound depends on
the integral of the square length of the mean curvature vector
field and the $n-$dimensional area of $M^n$.  It was shown in
\cite{PR} that the same formula does not work for any compact
spacelike surface in $\L^{4}$. However,

\begin{proposition}\label{ultima}
Let $\psi : M^2 \to \L^{\,4}$ be a compact spacelike immersion
which factors through the lightcone $\Lambda$. We have the
following inequality,
\begin{equation}\label{des10}
\lambda_{1}\leq 2\,\frac{\int_{M^2}\langle \mathbf{H}, \mathbf{H}\rangle \,dA}{\mathrm{area}(M^2,\langle\,,\,\rangle)},
\end{equation}
and equality holds if and only if $M^2$ is totally umbilical.
\end{proposition}
\begin{proof}
The inequality was obtained in \cite{PR} as a consequence of the
Hersch inequality \cite{Her} taking into account (\ref{traza}).
The equality holds in (\ref{des10}) if and only if $M^2$ has
constant Gauss curvature. Now, \cite[Theorem 5.4]{PR} ends the
proof.
\end{proof}

A compact spacelike surface $M^2$ in the $3$-dimensional de Sitter
space $\S^3_1$ with non-degenerate second fundamental form is
totally umbilical if and only if the Gauss curvature
$K^{\mathrm{II}}$ of its second fundamental form is constant,
\cite{AR}. A key tool in order to get this result is the Gauss
formula $K=1-\mathrm{det}(A)$ where $K$ and $A$ are the Gauss
curvature and the Weingarten operator of $M^2$, respectively. This
relationship permits to obtain a formula which relates $K$ and
$K^{\mathrm{II}}$ and involves different ingredients of
(\ref{des5}). This makes that the technique in \cite{AR} does not
work in order to show that a compact spacelike surface in
$\Lambda$, with $K^{\eta}$ a constant, must be totally umbilical.
Note that from Theorem \ref{des4}, this assertion is in fact
equivalent to the following one: if $K^{\eta}$ is a constant for
such a spacelike surface, then $K^{\eta}=2$. At the moment the
authors have no argument to support this assertion, although we
think that it holds true.

Note that $M^2$ compact, $\eta$ non-degenerate and $K^{\eta}$
constant imply $K^{\eta}\geq 2$. To prove this fact, take a point
$q_{0}\in M^2$ where the function $\mathfrak{d}$ attains its
maximum value. From Lemma \ref{des2} and Theorem \ref{nuevo} we
deduce that, $2K^{\eta}\geq K^{2}(q_{0}) /\mathfrak{d}(q_{0})\geq
4$.

\vspace{1mm} In view of the previous discussion, we state the
following

\vspace{1mm}

\noindent {\bf Conjecture.} {\it Every compact spacelike surface
in $\L^{4}$ which factors through the lightcone $\Lambda$ such
that $\eta$ is non-degenerate and $K^{\eta}=$constant must be
totally umbilical (that is, $K^{\eta}=2$).}

\end{document}